\documentclass[11pt, reqno]{amsart}

\usepackage{esint}
\usepackage{amstext}
\usepackage{amsthm}
\usepackage{amsmath}
\usepackage{amssymb}
\usepackage{latexsym}
\usepackage{amsfonts}
\usepackage{graphicx}
\usepackage{color}

\usepackage[mathscr]{euscript}

\usepackage[
hypertexnames=false, colorlinks, citecolor=black, linkcolor=blue, urlcolor=red]{hyperref}

\newcommand{\ci}[1]{_{ {}_{\scriptstyle #1}}}

{\end{list}}

\newcommand{\D}{\mathbb{D}}

\newcommand{\T}{\mathbb{T}}

\newcommand{\Z}{\mathbb{Z}}
\newcommand{\C}{\mathbb{C}}

\newcommand{\bI}{\mathbf{I}}
\newcommand{\bB}{\mathbf{B}}

\newcommand{\dd}{\mathrm{d}}

\newcommand{\cE}{\mathcal{E}}

\newcommand{\cH}{\mathcal{H}}
\newcommand{\cK}{\mathcal{K}}

\newcommand{\cV}{\mathcal{V}}

\newcommand{\cN}{\mathcal{N}}
\newcommand{\cP}{\mathcal{P}}

\newcommand{\fS}{\mathfrak{S}}
\newcommand{\fm}{\mathfrak{m}}

\newcommand{\fdot}{\,\cdot\,}
\newcommand{\wt}{\widetilde}

\newcommand{\bx}{\mathbf x}
\newcommand{\by}{\mathbf y}

\newcommand{\bO}{\mathbf{0}}

\newcommand{\bbeta}{\boldsymbol{\beta}}

\newcommand{\cC}{\mathcal{C}}
\newcommand{\bmu}{\boldsymbol{\mu}}

\newcommand{\rk}{\operatorname{rank}}

\newcommand{\ran}{\operatorname{Ran}}

\newcommand{\re}{\operatorname{Re}}
\newcommand{\tr}{\operatorname{tr}}

\newcommand{\cspn}{\overline{\operatorname{span}}}
\newcommand{\clos}{\operatorname{Clos}}

\newcommand{\f}{\varphi}

\newcommand{\fD}{\mathfrak{D}}

\newcommand{\ti}[1]{_{\scriptstyle \text{\rm #1}}}

\newcounter{vremennyj}

\numberwithin{equation}{section}

\newtheorem{thm}{Theorem}[section]
\newtheorem{lm}[thm]{Lemma}
\newtheorem{cor}[thm]{Corollary}

\newtheorem{prop}[thm]{Proposition}
\newtheorem*{prop*}{Proposition}

\theoremstyle{remark}
\newtheorem{rem}[thm]{Remark}
\newtheorem*{rem*}{Remark}

\begin{document}

\title[Preservation of a.c.~spectrum]{Preservation of absolutely continuous spectrum for contractive operators}

\author{Sergei Treil}
 \thanks{Work of S.~Treil is supported  in part by the National Science Foundation under the grant  
 	DMS-1856719.}
\address{Department of Mathematics \\ Brown University \\ Providence, RI 02912 \\ USA}
\email{treil@math.brown.edu}

\author{Constanze Liaw}
\address{Department of Mathematical Sciences\\ University of Delaware\\ Newark, DE 19716\\ USA\\and CASPER\\Baylor University\\
 Waco, TX  76798\\ USA}
\email{liaw@udel.edu}

\thanks{Work of C.~Liaw is supported  in part by the National Science Foundation under the grant  DMS-1802682. Since August 2020, C.~Liaw has been serving as a Program Director in the Division of Mathematical Sciences at the National Science Foundation (NSF), USA, and as a component of this position, she received support from NSF for research, which included work on this paper. Any opinions, findings, and conclusions or recommendations expressed in this material are those of the authors and do not necessarily reflect the views of the NSF}
\subjclass[2010]{Primary 47A55, 30H05, 47B32, 46E22. Secondary 30H10, 47B38}

\keywords{Trace class perturbations, contractive operators, dimension function, absolutely continuous spectrum}

\begin{abstract}
We consider contractive operators $T$ that are trace class perturbations of a unitary operator $U$. We prove that 
the dimension functions of the absolutely continuous spectrum of $T$, $T^*$ and of $U$ coincide. 
In particular, if $U$ has a purely singular spectrum then the characteristic function $\theta$ of $T$ is a \emph{two-sided inner} function, i.e.~$\theta(\xi)$ is unitary a.e.~on $\T$. 
Some corollaries of this result are related to investigations of the asymptotic stability of the operators $T$ and $T^*$ (convergence $T^n\to 0$ and $(T^*)^n\to 0$, respectively, in the strong operator topology). 

The proof is based on an explicit computation of the characteristic function.
\end{abstract}

\maketitle

\setcounter{tocdepth}{1}
\tableofcontents
\setcounter{tocdepth}{3}

\section*{Notation}

\begin{enumerate}
\item[$\D$\quad] The open unit disc in the complex plane $\C$, $\D:=\{z\in\C : |z|<1\}$. 
\item[$\T$\quad] The unit circle in $\C$, $\T=\partial\D$. 
\item[$\fm$\quad] The normalized ($\fm(\T)=1$) Lebesgue measure on $\T$. 
\item[ $\bI\ci\fD$, $\bI$\quad]   Identity operator; in most situations, where it is clear from the context we will skip the index, denoting the space where the operator acts. 
\item[$\fS_1$\quad] Trace class. 

\item[$\fS_2$\quad] The Hilbert--Schmidt  class.

\item[$H^2$\quad] The Hardy space $H^2$; we will also use the symbol $H^2(E)$ for the vector-valued $H^2$ functions with values in a Hilbert space $E$. 

\item[$z\to \xi\sphericalangle$\quad] $z\in\D$ approaches $\xi\in \T$ non-tangentially; the aperture of the non-tangential approach regions is assumed to be fixed (but not essential). 
\end{enumerate}

All Hilbert spaces in this paper are separable, and all operators act between Hilbert spaces (or on the same Hilbert space).  By a measure we always mean a finite Borel measure on $\T$. 

The term a.e.~always means a.e.~with respect to the Lebesgue measure on $\T$. For the a.e.~with respect to a different measure the term $\mu$-a.e.~is used. 

\section{Introduction and main results}
\label{s:intro}
Recall that a unitary operator $U$ on a separable Hilbert space is unitarily equivalent to  multiplication by the independent variable $\xi$ in the von Neumann direct integral of Hilbert spaces, 
\begin{align}
\label{e: direct integral 01}
\cN:= \int_\T \oplus E(\xi) \dd\mu(\xi).
\end{align}
The \emph{dimension function} $N(\xi)=N\ti U(\xi) := \dim E(\xi)$ is a unitary invariant of the operator $U$: together with the \emph{spectral type} $[\mu]$ of $\mu$, which is the class of all measures mutually absolutely continuous with $\mu$, they completely define the operator $U$ up to unitary equivalence. The function $N\ti U$ is often called the \emph{spectral multiplicity function}, and we will use this term. 

For the definiteness we assume that $N(\xi)=0$ whenever the the Lebesgue density $w=\dd \mu/\dd\fm$ of $\mu$ vanishes (note that $\mu(\{\xi\in\T:w(\xi)=0\})=0$).

The multiplicity of the absolutely continuous (a.c.) part of $U$, is by the definition, the function $N$ a.e.~with respect to the Lebesgue measure on $\T$.  

Let us now introduce the notion of spectral multiplicity (of the a.c.~spectrum) for a contraction. Recall that any contraction $T$ can be uniquely decomposed in the direct sum $T=V\oplus T_0$, where $T_0$ is a completely non-unitary (c.n.u.) contraction, and $V$ is unitary (either of these terms can be $0$). The spectral multiplicity of the a.c.~spectrum of $V$ is just the dimension function $N\ci{V}$ considered a.e.~with respect to  Lebesgue measure. 

As for the c.n.u.~part $T_0$, the rank of the defect functions $\Delta(\xi)$ and $\Delta_*(\xi)$, $\xi\in\T$ is often interpreted as the dimension functions for the a.c.~spectrum of a c.n.u.~contraction; in this paper we use this interpretation.

Let us recall the main definitions. Recall that a completely non-unitary contraction $T_0$  is uniquely defined (up to unitary equivalence) by its characteristic function $\theta=\theta\ci{T_0}$, cf.~\cite{SzNF2010}, which is an analytic operator-valued function on the unit disc $\D$, whose values are strict contractions $\theta(z): \fD\to \fD_*$; here $\fD$ and $\fD_*$ are some auxiliary Hilbert spaces. 

The characteristic function is defined up to constant unitary factors (possibly between different spaces) on both sides, so each such equivalence class corresponds to the collection of unitarily equivalent c.n.u.~contractions. We should also mention, that for a general contraction $T=V\oplus T_0$, its characteristic function coincides with the characteristic function of its purely contractive part $T_0$. 

Recall also that any bounded analytic function $F$ with values in $B(\fD;\fD_*)$ has non-tangential boundary values in the strong operator topology a.e.~on $\T$, and that $F(z)$, $z\in\D$ can be represented as the Poisson extensions of these boundary values. So for the characteristic function $\theta$ we denote its boundary values by $\theta(\xi)$, $\xi\in\T$, and we will treat $\theta$ as a function defined on $\D$ and a.e.~on $\T$. 

For a characteristic function $\theta$, its \emph{defect functions} are defined a.e.~on $\T$ as
\begin{align*}
\Delta:= (\bI - \theta^*\theta)^{1/2}, \qquad \Delta_*:= (\bI - \theta\theta^*)^{1/2}. 
\end{align*}
The functions $\rk \Delta(\xi)$ and $\rk \Delta_*(\xi)$, $\xi\in \T$ are often interpreted as the dimension functions for the a.c.~spectrum of a contraction, and we use this interpretation in the paper. 

\begin{thm}
\label{t: ac spectrum main}
Let $U$ be a unitary operator (on a separable Hilbert space), and let $K$ be a trace class operator such that $T=U+K$ is a contraction. If $T= V\oplus T_0$ is the decomposition of $T$ into unitary and completely non-unitary parts, and $\theta$ is the characteristic function of $T$, then 
\begin{align}\label{e:rank1}
\rk \Delta(\xi) &= \rk\Delta_*(\xi),  \\
\label{e:rank2}
N\ci U(\xi) &= N\ci{V}(\xi) + \rk \Delta(\xi)
\end{align}
a.e.~on $\T$. 
\end{thm}


We should mention that there is a large body of work studying the absolutely continuous spectrum in the case when the perturbed operator is not unitary/self-adjoint, see for example \cite{Naboko77, Naboko80, Naboko87, Tikho, Solomyak_B_scattering_1989}. However, these papers were mostly concerned with the existence of the wave operators, and we are not sure if it is possible to easily get our result from there. In particular our result covers the case when the spectrum of the perturbed operator is the whole closed unit disc, and a typical assumption in results about wave operators is the ``thinness" of the spectrum. 

Even if we assume that the spectrum is not the whole unit disc (for example if the unitary operator has purely singular spectrum \cite{Nik69}), a rigorous translation from one language to the other would be not much simpler than our self-contained presentation; and we would need to use some highly non-trivial the results from very technical papers.   

Corollaries \ref{c:1}, \ref{c:2}, \ref{c:3} below concern the asymptotic stability of the perturbed operator; 
some of these results might be known to experts. Let, like in Theorem \ref{t: ac spectrum main},  $T=U+K$ be a contraction, $U$  be unitary and $K\in\fS_1$. Let also $T=V\oplus T_0$ be the decomposition of $T$ into unitary and completely non-unitary parts, and let $\theta $ be the characteristic function of $T$. 

\begin{cor}\label{c:1}
If $U$ has purely singular spectrum (i.e.~if $\mu$ is purely singular), then $\theta$ is a double inner function, meaning that $\theta(\xi)$ is a unitary operator a.e.~on $\T$. 
\end{cor}

\begin{cor}\label{c:2}
If $U$ has purely singular spectrum, then $T_0$ and $T_0^*$ are asymptotically stable, meaning that $T_0^n\to 0$ and $(T_0^*)^n\to 0$ in the strong operator topology as $n\to \infty$. 
\end{cor}

\begin{cor}\label{c:3}
If $T$ is asymptotically stable, i.e.~if $T^n\to 0$ in the strong operator topology as $n\to \infty$, then $U$ has purely singular spectrum.  
\end{cor}



Note that in the corollaries above the spectrum of $T$ does not fill the unit disc (see \cite{Nik69} for Corollaries \ref{c:1}, \ref{c:2}, and \cite{TU83} for Corollary \ref{c:3}), so the results about wave operators can be used to prove the corollaries. 
However, as we mentioned above, such a proof relies on some highly technical  non-trivial papers. Moreover, we expect that  presented with all  the details it will not be significantly shorter than our self-contained paper. 

\begin{rem}
We should mention that under the assumptions of any of the above corollaries the operator $T_0$ belongs to the class of so-called $C_0$-contraction, meaning that there exists a function $\f\in H^\infty$ such that $\f(T_0)=0$. The theory for this operator class is well-developed, but not directly relevant for our paper, so, we will omit further discussion.
\end{rem}

Our proof of the main result (Theorem \ref{t: ac spectrum main}) is slightly lengthy but mostly elementary: after some simple operator-theoretic reasoning, we reduce everything to a particular case, see Lemma \ref{l: reduction 01} below.  We then express the characteristic function $\theta$ in terms of Cauchy--Herglotz transform of some $\fS_1$-valued measure, see Section \ref{s:char}. The proof of the theorem is then obtained by analyzing the boundary values of $\theta$, which is pretty straightforward, see Section \ref{s:proof main}.

We prove Corollaries \ref{c:1} through \ref{c:3} in Section \ref{ss:corollaries}.

\section{Some reductions}

Recall that an operator $T$ is called a \emph{strict contraction} if $\|Tx\|<\|x\|$ for all $x\ne 0$; clearly in this case $\|T\|\le 1$.

\begin{lm}
\label{l: reduction 01}
Let $T=U+K$, where $U$ is unitary, $K\in\fS_1$ and $\|T\|\le 1$ (all operators act on a Hilbert space $\cH$). Then $T$ can be represented as 
\begin{align}
\label{e: perturbation 01}
T = U_1 + \bB (\Gamma-\bI) \bB^* U_1, 
\end{align}
where $U_1$ is unitary, $U-U_1\in \fS_1$, $\bB:\fD\to\cH$ is an isometry from and auxiliary Hilbert space $\fD$, and $\Gamma=\Gamma^*\ge\bO$ is a strict contraction such that $\bI-\Gamma\in\fS_1$. 
\end{lm} 

In the proof of the above Lemma \ref{l: reduction 01} we will use the following trivial fact. 

\begin{lm}
\label{l: orthog 01} 
Let $\|T\|\le 1$, and let $\|Tx\|=\|x\|\ne0$. Then for any $y\perp x$ we have that $Ty\perp Tx$. 
\end{lm}

We leave the proof of this lemma as an exercise for the reader. 

We will need one more simple Lemma.

\begin{lm}
\label{l: polar decomposition I+S_1}
Let $R = \bI + K$, where $K\in\fS_1$.  Then one can write a polar decomposition   $R= V|R|$ where $|R|:=(R^*R)^{1/2}$ and $V$ is unitary, such that   $|R|-\bI \in\fS_1$ and $V-\bI \in\fS_1$. 
\end{lm}
\begin{rem*}
The term $|R|$  in the polar decomposition is uniquely defined. The unitary operator $V$ is 
uniquely defined if and only if $\ker R=\{0\}$ (since $R$ is a Fredholm operator of index $0$, it 
happens if and only if $\ran R$ is the whole space). Formally the above Lemma \ref{l: polar 
decomposition I+S_1} means that for some choice of the unitary operator $V$ we have $\bI -V\in 
\fS_1$; while it is not essential for the proof, one can see from the proof, that in fact, 
\emph{all} possible choices of $V$ satisfy $\bI-V\in\fS_1$. 
\end{rem*}

\begin{proof}[Proof of Lemma \ref{l: polar decomposition I+S_1}]
Let us first consider the case when $R$ is invertible (which happens if and only if $\ker R=\{0\}$). In this case, $|R|$ is trivially invertible and $V$ is unique and is defined as $V=R|R|^{-1}$. 

We know that $R\in\bI + \fS_1$, so trivially $|R|^2 = R^*R\in \bI + \fS_1$, and so $|R|\in\bI+\fS_1$. Since $|R|$ is invertible, it is easy to see that $|R|^{-1}\in \bI+\fS_1$, and therefore $V= R|R|^{-1}\in \bI + \fS_1$. 

Now, let us consider the general case. Since for a compact $K$ the operator  $\bI+K$ is  Fredholm of index $0$, the range of $R$ is closed, and $\dim \ker R = \dim \ker R^* <\infty$ (and $\ran R =(\ker R^*)^\perp$). Take any invertible operator $R_1:\ker R \to \ker R^*$ (such an operator exists and has finite rank, because $\dim \ker R = \dim \ker R^* <\infty$). Define $\wt R:= R+R_1$. By the construction, $\wt R$ is invertible, and maps $(\ker R)^\perp $ onto $\ran R=(\ker R^*)^\perp$ and $\ker R$ onto $\ker R^*$. 

Note also that $\wt R -\bI\in \fS_1$. 

If we denote by $R_0$ the restriction of $R$ onto $(\ker R)^\perp$ (with target space restricted to 
$\ran R = (\ker R^*)^\perp$), we can see that $|\wt R|$ in the decomposition $ (\ker R)^\perp\oplus 
\ker R$ has the block diagonal form 
\begin{align}
\label{e: block wt R}
|\wt R| = \left( 
\begin{array}{cc}
|R_0| & 0 \\
0 & |R_1|
\end{array}  \right) . 
\end{align}
Consider the polar decomposition $\wt R = V |\wt R|$; since $\wt R$ is invertible, $V$ is uniquely 
defined by $V=\wt R |\wt R|^{-1}$. As we discussed above in the beginning of the proof, since $\wt 
R$ is invertible, we have that $V\in \bI +\fS_1$. We also know that $|\wt R| \in \bI+\fS_1$, and so 
$|R|\in\bI+\fS_1$, because $|R|$ differs from $|\wt R|$ by a finite rank block $|R_1|$, see 
\eqref{e: block wt R} above. 

The fact that $\wt R$ maps $(\ker R)^\perp $ onto $\ran R=(\ker R^*)^\perp$ and $\ker R$ onto $\ker R^*$ and the block diagonal structure \eqref{e: block wt R} imply that $V$ also maps $(\ker R)^\perp $ onto $\ran R=(\ker R^*)^\perp$ and $\ker R$ onto $\ker R^*$. Therefore $R=V|R|$, so we have constructed the desired polar decomposition. 
\end{proof}

\begin{proof}[Proof of Lemma \ref{l: reduction 01}]
We will prove a ``dual'' formula to \eqref{e: perturbation 01}, namely the formula 
\begin{align}
\label{e: perturbation 02}
T = U_1 + U_1 \bB (\Gamma-\bI) \bB^*;
\end{align}
applying this formula to the adjoint $T^* = U^* + K^*$ and then taking the adjoint we will get \eqref{e: perturbation 01}. 

The identity $U+K = U(\bI + U^*K)$ means that it is sufficient to prove \eqref{e: perturbation 02} for the particular case $U=\bI$.  

So, let $T=\bI+K$, $\|T\|\le 1$ and $K\in\fS_1$. 
Denote $\fD_1:= (\ker K)^\perp$.  Clearly
\begin{align*}
(\bI + K) x & = x \qquad \forall x\in \fD_1^\perp,   
\end{align*}
and therefore by Lemma \ref{l: orthog 01}
\begin{align*}
(\bI + K)\fD_1 & \subset \fD_1. 
\end{align*}
Then $K\fD_1\subset \fD_1$, and we can treat $K$ as an operator on $\fD_1$. 

So, let us restrict our attention to $\fD_1$. Denote $R=(\bI+K)\Bigm|_{\fD_1}$.

By Lemma \ref{l: polar decomposition I+S_1} we can write a polar decomposition  $R= V|R|$ of $R$, with unitary $V$  such that 
\begin{align*}
|R|-\bI\ci{\fD_1} \in \fS_1, \qquad V-\bI\ci{\fD_1} \in \fS_1 . 
\end{align*}
Denote $\fD_2:= \fD_1\ominus \ker (|R| - \bI\ci{\fD_1})$. Then trivially, $\Gamma= |R|\Bigm|_{\fD_2}$ is a strict contraction on $\fD_2$, $\Gamma=\Gamma^*$, and $\Gamma-\bI\ci{\fD_2}\in \fS_1$.  

Now gathering everything together we see that 
\begin{align*}
\bI\ci\cH + K = U_1 + U_1 P\ci{\fD_2} (\Gamma - \bI\ci{\fD_2}) P\ci{\fD_2}
\end{align*}
where 
\begin{align*}
U_1x = 
\begin{cases}
V x  & x\in \fD_1 \\
x  & x\in \cH \setminus \fD_1 . 
\end{cases}
\end{align*}
Note, that $\fD_1$ is a reducing subspace for $U_1$, and so the condition $V-\bI\ci{\fD_1}\in \fS_1$ implies that $U_1-\bI\ci\cH \in\fS_1$. Thus we have proved the formula \eqref{e: perturbation 02} for the case $U=\bI$ (with $\fD=\fD_2$ and $\bB$ being the embedding of $\fD_2$ into $\cH$). 

If $\fD$ is an abstract space, $\dim\fD=\dim\fD_2$, then taking an isometry $\bB: \fD\to \cH$, $\ran \bB= \fD_2$,  we can rewrite the above identity as 
\begin{align*}
\bI\ci\cH + K = U_1 + U_1 \bB (\bB^*\Gamma \bB - \bI\ci{\fD}) \bB^*
\end{align*}
so the general ``abstract'' form of \eqref{e: perturbation 02} is proved for $U=I$. 

As we discussed in the beginning of the proof, this proves \eqref{e: perturbation 02} for the general case, and so the formula \eqref{e: perturbation 01}. Lemma \ref{l: reduction 01} is proved. 
\end{proof}

\section{Characteristic functions}\label{s:char}

\subsection{Operator-valued spectral measures and spectral representation}
An operator-valued measure $\bmu$ on $\T$ is a countably additive  function defined on Borel subsets of $\T$ with values in the set of non-negative self-adjoint operators. This definition means that an operator-valued measure is always \emph{finite}, i.e.~that $\bmu(\T)$ is a bounded operator. 

Let $U$ be a unitary operator on $\cH$  and let and  operator $\bB:\fD \to \cH$ have trivial kernel, and let $\ran \bB$ be star-cyclic for $U$. Define the operator-valued spectral measure  $\bmu =\bmu\ci U$ (with values in $B(\fD)$) as 
\begin{align}
\label{e: bmu 01}
\bmu (E) = \bB^* \cE (E) \bB, 
\end{align}
for any Borel  $E\subset \T$; here $\cE=\cE\ci U$ is the (projection-valued) spectral measure of $U$. An equivalent definition is that $\bmu$ is the unique operator-valued measure such that 
\begin{align*}
\bB^*   U^n \bB = \int_\T \xi^n\dd\bmu(\xi) \qquad \forall n\in \Z, 
\end{align*}
or equivalently, 
\begin{align}
\label{e: bmu 03}
\bB^* (\bI - z U^*)^{-1}\bB = \int_\T \frac1{1-z\overline\xi}\dd\bmu(\xi) =: \cC\bmu(z)  \qquad \forall z\in\C\setminus\T. 
\end{align}
The operator $U$ is unitarily equivalent to the multiplication operator $M_\xi$ by the independent variable $\xi$ in the weighted space $L^2(\bmu)$. 

Let us recall that the weighted space $L^2(\bmu)$ with the operator-valued measure $\bmu$ is defined as follows.  First the inner product in $L^2(\bmu)$ is introduced on  functions of the form $f=\f \bx$, where $\f$ is a scalar-valued measurable function and $\bx\in \fD$:
\begin{align*}
\left( \f \bx, \psi\by \right)\ci{L^2(\bmu)} := \int_\T \f(\xi) \overline{\psi(\xi)}\left(\dd\bmu(\xi) \bx, \by \right)\ci{\fD}.
\end{align*}
This inner product is then extended by linearity to the set of all (finite) linear combinations of such functions. Such linear combinations (of course,  modulo the class of functions of norm $0$) form an inner product space, and its completion is, by definition, the weighted space $L^2(\bmu)$. 

The unitary operator $\cV: \cH\to L^2(\bmu)$ such that $M_\xi = \cV U \cV^*$ is also well-known. 
Namely, for $x\in\fD$, 
\begin{align*}
\cV [\f(U)\bB x ]= \f(\fdot)  x \in L^2(\bmu). 
\end{align*}

\subsection{Trace class operator-valued measures, spectral representation and spectral multiplicity function}
\label{s: trace class measures}
The representation of a unitary operator as a multiplication operator in the weighted space $L^2(\bmu)$ with operator-valued measure looks like ``abstract nonsense'', and the model looks more complicated than the original object. However, when the measure $\bmu$ takes values in the set $\fS_1$ of trace class operators, all objects are significantly simplified. 

If $\bmu$ is taking values in the set $\fS_1$ of trace class operators, we can define the scalar-valued measure $\bmu$ as $\mu:=\tr \bmu$. In this case the operator-valued measure $\bmu$ can be represented as 
\begin{align*}
\dd\bmu = W \dd\mu, 
\end{align*}
where $\| W(\xi)\|\le \| W(\xi)\|\ci{\fS_1} =1$ $\mu$-a.e.~on $\T$. 

It is not hard to see that in this case the measure $\mu=\tr\bmu$ is a scalar spectral measure of the operator $U$ that can be used in the von Neumann direct integral \eqref{e: direct integral 01}. The inner product in  weighted space $L^2(\bmu)$  can be computed (for  measurable functions $f$ and $g$) as   
\begin{align*}
\left( f, g\right)\ci{L^2(\bmu)} = \int_\T \bigl( W(\xi) f(\xi), g(\xi)  \bigr)\ci{\fD} \dd\mu(\xi) . 
\end{align*}

The weighted space $L^2(\bmu)$ in this case consists of all measurable functions for which $\|f\|\ci{L^2(\bmu)}<\infty$ (taking the obvious quotient space over the set of functions of norm $0$). 

It is also not hard to see that in this case the dimension function $N\ci U(\xi)$ can be computed as 
\begin{align*}
N\ci U(\xi) = \rk W(\xi)  , \qquad \mu\text{-a.e.}
\end{align*}
(recall that we assume that $\ran \bB$ is star-cyclic for $U$).

We presented just the fact we  will need; an interested reader can find more details in \cite{Kuroda1967}.

\subsection{Characteristic function via Cauchy--Herglotz integral of a trace class measure}
Recall, that we defined the following Cauchy transforms, 
\begin{align*}
\cC\bmu (z) := \int_\T \frac{\dd \bmu(\xi)}{1-z\overline\xi}, \qquad 
\cC_1\bmu (z) := \int_\T \frac{z\overline\xi}{1-z\overline\xi} \dd \bmu(\xi), \qquad 
\cC_2\bmu (z) := \int_\T \frac{1+z\overline\xi}{1-z\overline\xi}\dd \bmu(\xi) . 
\end{align*}

In \cite{Liaw-Treil_APDE_2019} we obtained the following formula for the characteristic function $\theta=\theta\ci\Gamma$ of the operator 
\begin{align}
\label{e: T_Gamma}
T\ci\Gamma = U + \bB (\Gamma-\bI) \bB^*U. 
\end{align}
where $\bB$ is an isometry acting from $\fD\to \cH$ and  $\Gamma$ (and therefore $\Gamma^*$) is a 
strict contraction. 

Recall that for a contraction $\Gamma$ the \emph{defect operator} $D\ci\Gamma$ is defined by $D\ci\Gamma:= (\bI -\Gamma^*\Gamma)^{1/2}$. 

The characteristic function $\theta=\theta\ci\Gamma$ of $T\ci\Gamma$  was proved to be given by 
\begin{align}
\label{e: char funct F1 01}
\theta\ci{\Gamma} (z)  &= -\Gamma  +   D\ci{\Gamma^*} F_1(z) \Bigl( \bI -(\Gamma^*-\bI)F_1(z)\Bigr)^{-1}  D\ci\Gamma
\\
\label{e: char funct F1 02}
&= -\Gamma  +   D\ci{\Gamma^*}  \Bigl( \bI -F_1(z)(\Gamma^*-\bI)\Bigr)^{-1} F_1(z)  D\ci\Gamma,
\end{align}
where $F_1(z)=\cC_1\bmu(z)$, and $D\ci\Gamma$ and $D\ci{\Gamma^*}$ are the defect operators. Here the measure $\bmu$ was given by \eqref{e: bmu 01}, or equivalently by \eqref{e: bmu 03}. 

This formula was proved in \cite{Liaw-Treil_APDE_2019} for the case of finite rank perturbations. However, the only place, where the finite rank was used in the proof was in the definition of the measure $\bmu$, which was in that case expressed explicitly via the scalar spectral measure in the von Neumann direct integral \eqref{e: direct integral 01} and the matrix of the operator $\bB$. Such explicit expression is not possible in the general case, but what one really needs for the proof of the formula, is the identity
\begin{align}\label{e:F1}
z {\bB}^*(\bI\ci \cH - z U^*)^{-1}U^*{\bB} = 
\cC_1\bmu(z)=: F_1(z). 
\end{align}

For convenience, we include the proof of \eqref{e: char funct F1 01}, \eqref{e: char funct F1 02} and \eqref{e:F1} in Section \ref{s:APP}.

Moving forward, we would like to express the characteristic function $\theta$ in terms of the Cauchy--Herglotz integral $\cC_1\wt\bmu$ of some $\fS_1$-valued measure $\wt\bmu$: this will allow us to express the defect functions $\Delta := \left( \bI - \theta^*\theta  \right)^{1/2}$ and $\Delta_* := \left( \bI - \theta\theta^*  \right)^{1/2}.$

First, let us express $\theta$ in terms of $F_2 :=\cC_2\bmu$. Using the fact that $\Gamma=\Gamma^*$ we can rewrite \eqref{e: char funct F1 01} as 
\begin{align*}
\theta\ci\Gamma &= D\ci{\Gamma} \left( -  D\ci{\Gamma}^{-1} \Gamma D\ci{\Gamma}^{-1}  + F_1 \left(\bI - (\Gamma -\bI)F_1 \right)^{-1} \right) D \ci{\Gamma}  \\
&= D\ci{\Gamma}^{-1}   \left( -  \Gamma +  \Gamma (\Gamma -\bI)F_1  + D\ci{\Gamma}^2 F_1\right) \left(\bI - (\Gamma -\bI)F_1 \right)^{-1}  D \ci{\Gamma} ;
\end{align*}
note that while the operator $D\ci\Gamma^{-1}$ is unbounded, it is densely defined, and the above 
identity can be understood as an identity for bilinear forms on a dense linear submanifold 
$\fD\times \ran D\ci\Gamma \subset \fD\times\fD$.

Since $F_1= (F_2-\bI)/2$ we can continue
\begin{align*}
\theta\ci\Gamma &= D\ci{\Gamma}^{-1}   \Bigl( -2 \Gamma +   ( \bI -\Gamma)(F_2 -\bI) \Bigr) \Bigl(2\bI - (\Gamma -\bI)(F_2 -\bI) \Bigr)^{-1}  D \ci{\Gamma}  
\\
&=  D\ci{\Gamma}^{-1}   \Bigl( -(\bI + \Gamma) +   ( \bI -\Gamma)F_2  \Bigr) (\bI -\Gamma)D\ci{\Gamma}^{-1} \\
&\qquad\qquad\qquad\qquad
\left[ D \ci{\Gamma}^{-1} \Bigl( (\bI + \Gamma) - (\Gamma -\bI)F_2 \Bigr)  (\bI -\Gamma)D\ci{\Gamma}^{-1} \right]^{-1};
\end{align*}
note that for a strict contraction $\Gamma=\Gamma^*\ge \bO$
the operator $(\bI -\Gamma)D\ci{\Gamma}^{-1} = (\bI -\Gamma)^{1/2} (\bI +\Gamma)^{-1/2} $ is bounded, so the above expression is again well defined. 

Using the identity 
\begin{align*}
D\ci\Gamma^{-1}(\bI - \Gamma)(\bI + \Gamma)D\ci\Gamma^{-1}=\bI
\end{align*}
we can rewrite  $\theta\ci\Gamma  $ as 
\begin{align*}
\theta\ci\Gamma  & = 
\left( \bbeta F_2 \bbeta - \bI \right) \left( \bbeta F_2 \bbeta + \bI \right)^{-1}, 
\end{align*}
where
\begin{align*}
\bbeta = \bbeta^* : = D\ci{\Gamma}^{-1} (\bI - \Gamma).  
\end{align*}
Define the measure $\wt\bmu := \bbeta \bmu\bbeta$. Then, trivially, $\cC_2\wt\bmu = \bbeta \cC_2 \bmu \bbeta = \bbeta F_2  \bbeta$, so 
\begin{align}
\label{e: theta 04}
\theta = \frac{\cC_2\wt\bmu-\bI}{\cC_2\wt\bmu+\bI}
\end{align}
(we write it as a fraction to emphasize that the terms commute).   Note that for $z\in\D$ we have $\re (\cC_2\wt\bmu (z) +\bI) \ge\bI$, so the operator $\cC_2\wt\bmu (z) +\bI$ is invertible and the right hand side of \eqref{e: theta 04} is well defined. 

Finally, under our assumptions that $\Gamma=\Gamma^*$ is a strict contraction and $\bI-\Gamma\in \fS_1$, the formula
\begin{align*}
\bbeta= D\ci{\Gamma}^{-1} (\bI - \Gamma) = (\bI -\Gamma)^{1/2} (\bI +\Gamma)^{-1/2}
\end{align*}
shows that $\bbeta\in\fS_2$ (the Hilbert--Schmidt class). Therefore, the measure $\wt\bmu$ is $\fS_1$-valued. 

\section{Proof of main results}
\label{s:proof main}
\subsection{Proof of Theorem \ref{t: ac spectrum main}: the principal case}\label{ss-ProofTHM1.1}
In this subsection we prove Theorem \ref{t: ac spectrum main} for the main special case when
\begin{align*}
T= U +\bB (\Gamma-\bI)\bB^* U, 
\end{align*}
where $\ran \bB$ is star-cyclic for $U$ and $\Gamma=\Gamma^*\ge\bO$ is a strict contraction, $\bI-\Gamma\in\fS_1$. The general case can be easily obtained from it using Lemma \ref{l: reduction 01}, see Section \ref{s: Proof general case} below. 

\subsubsection{Some technical lemmas}
The function $\theta$ is defined in the open unit disc $\D$. Since it is a bounded analytic operator-valued function, it possess non-tangential boundary values 
\begin{align*}
\theta(\xi):= \lim_{z\to\xi\sphericalangle}\theta(z), \qquad \xi\in \T
\end{align*}
(in the strong operator topology) a.e.~on $\T$. 

\begin{lm}
\label{l: I-theta invertible ae}
The  function $\bI - \theta$ is invertible a.e.~on $\T$.
\end{lm}

\begin{proof}
One can see from \eqref{e: theta 04} that for $z\in\D$
\begin{align}
\label{e: I - theta invertible}
 (\bI-\theta(z))^{-1} = (\cC_2\wt\bmu(z) + \bI)/2 = \cC \wt\bmu(z). 
\end{align}
As we discussed at the end of Section \ref{s:char}, the measure $\wt\bmu$ is $\fS_1$-valued, so it can be represented as 
\begin{align*}
\dd\wt\bmu = W\dd\mu, 
\end{align*}
where the scalar measure $\mu$ is given by $\mu = \tr \wt\bmu$. In this case 
\begin{align*}
\|W(\xi)\|\ci{\fS_2} \le \|W(\xi)\|\ci{\fS_1} = 1 \qquad \mu\text{-a.e.~on }\T. 
\end{align*}
The space $\fS_2$ is a Hilbert space, so by Lemma \ref{l: Cauchy boundary values Hilbert} below the non-tangential boundary values of $\cC\wt\bmu$ exist a.e.~on $\T$. Note that for our purposes it is sufficient that the boundary values exist in the strong operator topology, while  Lemma \ref{l: Cauchy boundary values Hilbert} states that the boundary values exist in the (much stronger) topology of  $\fS_2$. 

The equality \eqref{e: I - theta invertible} means that for all $z\in\D$
\begin{align*}
\cC \wt\bmu(z) (\bI-\theta(z))      =   (\bI-\theta(z)) \cC \wt\bmu(z) = \bI, 
\end{align*}
and taking the non-tangential boundary values we conclude that the same identities hold a.e.~on $\T$. But this exactly means that $\bI - \theta$ is invertible a.e.~on $\T$. 
\end{proof}

\begin{lm}
\label{l: Cauchy boundary values Hilbert}
Let $\mu$ be a (finite) Borel measure on $\T$, and let $f\in L^2(\mu; E)$  (where $E$ is a Hilbert space). Then the non-tangential boundary values 
\begin{align*}
[\cC f\mu](\xi) = \lim_{z\to\xi\sphericalangle} [\cC f\mu](z) , \qquad \xi\in \T
\end{align*}
(in the norm topology of $E$) exist a.e.~on $\T$.  
\end{lm}
\begin{proof}
We use the following well-known result (\cite[Theorem 1.1]{Aleksandrov_Zap_1989}, see also \cite[Proposition 10.2.3]{cimaross})  that for a measure $\mu$ the operator $\cV_\mu$, 
\begin{align*}
\cV_\mu f (z) := \frac{[\cC f\mu] (z)}{[\cC\mu] (z)}, \qquad f\in L^2(\mu)
\end{align*}
is a bounded operator from $L^2(\mu)$ to the Hardy space $H^2$, $\|\cV_\mu\|\ci{L^2(\mu)\to H^2} \le C(\mu)$.%
\footnote{In fact, it is well-known and not hard to show that for a probability measure $\mu$ the operator $\cV_\mu$ is a contraction. Simple scaling then allows one to get the estimate $\|\cV_\mu\|\ci{L^2(\mu)\to H^2} \le \mu(\T)^{-1/2}$.} 

The operator $\cV_\mu$ is defined on scalar-valued functions, but the same formula defines an operator on the vector-valued space $L^2(\mu; E)$. 
Take $f\in L^2(\mu;E)$. Applying the scalar estimate to each coordinate of $f$, we conclude that 
\begin{align*}
\|\cV_\mu f \|\ci{H^2(E)} \le C(\mu) \| f\|\ci{L^2(\mu;E)} \qquad \forall f\in L^2(\mu;E). 
\end{align*}
It is well-known that for $g\in H^2(E)$ the non-tangential boundary values (in the norm topology of $E$) exist a.e.~on $\T$. 
It is also well-known that (finite and non-zero) non-tangential boundary values of $\cC\mu$ exist a.e.~on $\T$. Since for $f\in L^2(\mu;E)$
\begin{align*}
[\cC f\mu ] (z) = \cV_\mu f(z)/[\cC\mu](z), 
\end{align*}
we immediately get the conclusion of the lemma. 
\end{proof}

\subsubsection{Computing the defect functions}
Recall that the spectral measure $\wt\bmu$ is represented as $\dd\wt\bmu = W \dd\mu$, where $\mu=\tr\wt \bmu$. 
Denote by $w$ the Lebesgue density of $\mu$ (i.e.~of its absolutely continuous part), $w:=\dd \mu/\dd \fm$.  
\begin{prop}
\label{p: Deltas}
The defect functions $\Delta$ and $\Delta_*$ can be computed as
\begin{align}
\label{e: Delta}
\Delta(\xi)^2 & = (\bI -\theta(\xi)^*) W(\xi)w(\xi) (\bI - \theta(\xi)), \\
\label{e: Delta*}
\Delta_*(\xi)^2 & = (\bI -\theta(\xi)) W(\xi)w(\xi) (\bI - \theta(\xi)^*)
\end{align}
a.e.~on $\T$. 
\end{prop}

\begin{proof}
Let $\cP\wt\bmu$ be the Poisson extension of the measure $\wt\bmu$. Trivially, 
\begin{align*}
\cP\wt\bmu = \re \wt F_2, 
\end{align*}
where $\wt F_2 = \cC_2\wt\bmu$. The representation $\dd\wt\bmu = W \dd\mu$, implies that the non-tangential boundary values of $\cP \wt\bmu$ exist and coincide with $Ww$ a.e.~on $\T$; the non-tangential boundary values exist a.e.~in the $\fS_2$ norm, cf.~Lemma \ref{l: Cauchy boundary values Hilbert} above, but for our purposes taking limits in the strong operator topology will be enough.    

So, we have
\begin{align*}
\cP\wt\bmu = \re \wt F_2 = \re [(\bI + \theta)(\bI-\theta)^{-1}  ].
\end{align*}
Computing we get (for $z\in\D$)
\begin{align*}
\cP\wt\bmu 
&=
\re[(\bI+\theta)(\bI-\theta)^{-1}]
=
\frac{1}{2} [(\bI+\theta)(\bI-\theta)^{-1}+(\bI-\theta^*)^{-1}(\bI+\theta^*)]\\
&=
\frac{1}{2} (\bI-\theta^*)^{-1}\left[(\bI-\theta^*)(\bI+\theta)+(\bI+\theta^*)(\bI-\theta)\right](\bI-\theta)^{-1}\\
&=
\frac{1}{2} (\bI-\theta^*)^{-1}[2\bI -2 \theta^*\theta](\bI-\theta)^{-1}
\\
&=(\bI-\theta^*)^{-1}[\bI - \theta^*\theta](\bI-\theta)^{-1}.
\end{align*}
Taking the non-tangential boundary values we get that 
\begin{align*}
Ww = (\bI-\theta^*)^{-1}[\bI - \theta^*\theta](\bI-\theta)^{-1} = (\bI-\theta^*)^{-1}\Delta^2(\bI-\theta)^{-1}
\end{align*}
a.e.~on $\T$. Since the function $\bI-\theta$ is invertible a.e.~on $\T$ by Lemma \ref{l: I-theta invertible ae}, this identity is equivalent to \eqref{e: Delta}. 

To get \eqref{e: Delta*} we just need to repeat the above calculation with the order of $\theta$ and $\theta^*$ interchanged, namely
\begin{align*}
\cP\wt\bmu 
&=
\re[(\bI+\theta)(\bI-\theta)^{-1}]
=
\frac{1}{2} [(\bI-\theta)^{-1}(\bI+\theta)+(\bI+\theta^*) (\bI-\theta^*)^{-1}]\\
&=
\frac{1}{2} (\bI-\theta)^{-1}\left[ (\bI+\theta)(\bI-\theta^*)  + (\bI-\theta)(\bI+\theta^*)\right](\bI-\theta^*)^{-1}\\
&=
\frac{1}{2} (\bI-\theta)^{-1}[2\bI -2 \theta\theta^*](\bI-\theta^*)^{-1}
\\
&=(\bI-\theta)^{-1}[\bI - \theta\theta^*](\bI-\theta^*)^{-1}.
\end{align*}
Taking boundary values again, we get that 
\begin{align*}
Ww = (\bI-\theta)^{-1}\Delta_*^2(\bI-\theta^*)^{-1}
\end{align*}
which is equivalent to \eqref{e: Delta*} because $\bI -\theta$ is invertible a.e.~on $\T$. 
\end{proof}

\subsubsection{Completion of the proof of the principal case}
As we discussed above in Section \ref{s: trace class measures} the dimension function $N$ can be computed as $N(\xi) = \rk W(\xi)$ $\mu$-a.e. By Lemma \ref{l: I-theta invertible ae} the function $\bI-\theta$ is invertible a.e.~on $\T$, so the conclusion of Theorem \ref{t: ac spectrum main} (in the case we are considering)  immediately follows from identities  \eqref{e: Delta}, \eqref{e: Delta*}.  \hfill\qed

\subsection{Proof of Theorem \ref{t: ac spectrum main}: general case}
\label{s: Proof general case}
According to Lemma \ref{l: reduction 01} the operator $T$ can be represented as 
\begin{align*}
T= U_1 +\bB (\Gamma-\bI)\bB^* U_1
\end{align*}
where, as in Section \ref{ss-ProofTHM1.1}, we have $U-U_1\in \fS_1$, $\Gamma=\Gamma^*$ is a strict contraction, and $\bI-\Gamma\in\fS_1$.  Note that $\ran \bB$ is not necessarily star-cyclic for $U$. 
Denote $\cH_0:= \cspn\{U_1^n \ran \bB: n\in\Z  \}$, $\cH_1:= \cH_0^\perp$. The subspaces $\cH_0$, $\cH_1$ are reducing for both $U_1$ and $T$; moreover $T|_{\cH_1} = U_1|_{\cH_1}$ is trivially unitary and $T|_{\cH_0}$ is a completely non-unitary contraction on $\cH_0$, see \cite[Lemma 1.4]{Liaw-Treil_APDE_2019}. 

Denote $V:= U_1|_{\cH_1}$, $U_0:= U_1|_{\cH_0}$, $T_0:= T|_{\cH_0}$ (with the target space also restricted to the spaces $\cH_0$, $\cH_1$ respectively). Clearly
\begin{align*}
T_0 = U_0 +\bB (\Gamma-\bI)\bB^* U_0
\end{align*}
and $\ran\bB$ is star-cyclic for $U_0$. 
Recall that, as we discussed in Section \ref{s:intro}, the characteristic function of the contraction $T$ coincides with the characteristic function of its completely non-unitary part $T_0$. 
Therefore, by the discussion in Section \ref{ss-ProofTHM1.1} 
\begin{align*}
\rk \Delta(\xi) = \rk \Delta_*(\xi) = N\ci{U_0}(\xi)\qquad \text{a.e.~on }\T , 
\end{align*}
which gives us \eqref{e:rank1}. 
Adding $N\ci V(\xi)$ to both parts, and noticing that $N\ci V(\xi) + N\ci{U_0}(\xi) = N\ci{U_1}(\xi)$, we get 
that
\begin{align*}
\rk V(\xi) + \rk \Delta(\xi) 
= N\ci{U_1}(\xi)\qquad \text{a.e.~on }\T  .
\end{align*}
The above formula is exactly the identity \eqref{e:rank2} with $N\ci{U_1}(\xi)$ instead of $N\ci{U}(\xi)$. But $U-U_1\in\fS_1$, so by the classical Kato--Rosenblum theorem (or, more precisely Birman--Krein theorem%
\footnote{The statement that  we are using, about the preservation of the absolutely continuous parts of a unitary trace class perturbations of unitary operators, first appeared in \cite{Birman-Krein_DAN_1962}, which should be a proper reference. It is also can be obtained via linear fractional transformation from an appropriate version (difference of resolvents is trace class) of the Kato--Rosenblum theorem for self-adjoint operators due to S.~T.~Kuroda \cite{Kuroda1959, Kuroda1960}.} %
 \cite{Birman-Krein_DAN_1962}) the identity 
$N\ci U(\xi) = N\ci{U_1}(\xi)$ holds a.e.~on $\T$, so \eqref{e:rank2} holds. Thus Theorem \ref{t: ac spectrum main} is proved in full generality.
\hfill\qed


\subsection{Proof of the corollaries 
}\label{ss:corollaries}

As in Theorem \ref{t: ac spectrum main}, let $U$ be a unitary operator, let $K\in \fS_1$ and $T=U+K$. Further let $\theta$ be its characteristic function and let $T=V\oplus T_0$ be the decomposition of $T$ into unitary and c.n.u.~parts. 

Recall that a bounded analytic  operator-valued function $\theta$  on the unit disc is called  \emph{inner} if its boundary values $\theta(\xi)$ are isometries a.e.~on $\T$. The function $\theta$ is called $*$-inner (or co-inner) if the function $z\mapsto \theta(\overline z)^*$ is inner, which means that operators $\theta(\xi)^*$ are isometries a.e.~on $\T$. 

Finally, the function $\theta$ is called \emph{double inner} if it is both inner and $*$-inner, which means that boundary values $\theta(\xi)$ are unitary a.e.~on $\T$. 

\begin{proof}[Proof of Corollary \ref{c:1}]
Let $U$ have purely singular spectrum, which means that $N\ci U(\xi) = 0$ a.e.~on $\T$. Then equation \eqref{e:rank2} informs us that $ \rk\Delta(\xi) = 0$ a.e.~on $\T$, and by \eqref{e:rank1}, we obtain that also  $\rk\Delta_*(\xi) = 0$ a.e.~on $\T$. Therefore, by  the definition of the  defect functions $\Delta$ and $\Delta_*$, we have that $\theta(\xi)$ is unitary a.e.~on $\T$, i.e.~that $\theta$ is double inner.
\end{proof}

For the proofs of Corollaries \ref{c:2} and \ref{c:3}, recall the following result.

\begin{prop}[see, e.g.~{\cite[Proposition VI.3.5]{SzNF2010}}]\label{p:SzNF}
For a c.n.u.~contraction $T_0$ we have:
\begin{itemize}
    \item[(i)] $T_0$ is asymptotically stable if and only if its characteristic function $\theta$ 
    is $*$-inner. 
    \item[(ii)] $T_0^*$ is asymptotically stable if and only if its characteristic function 
    $\theta$ is inner. 
\end{itemize}
\end{prop}

\begin{proof}[Proof of Corollary \ref{c:2}]
This follows immediately from Corollary \ref{c:1} and the ``if'' direction of both items in 
Proposition \ref{p:SzNF}.
\end{proof}

\begin{proof}[Proof of Corollary \ref{c:3}]
Let $T$ be asymptotically stable.

Then its unitary part is trivial ($V=0$). In particular, the dimension function of its absolutely continuous part $N\ci{V}(\xi)$ is trivial, i.e.~$N\ci{V}(\xi)=0$ a.e.~$\xi\in \T.$

From the asymptotic stability of $T$ we further obtain that $T$ is a c.n.u.~contraction. In particular, $T=T_0$ is asymptotically stable. So Proposition \ref{p:SzNF} implies that $\theta$ is $*$-inner. Therefore, we have $\rk \Delta_*(\xi)=0$ a.e.~$\xi\in\T$ and so by \eqref{e:rank1} $\rk \Delta(\xi)=0$ a.e.~$\xi\in\T$.

Invoking \eqref{e:rank2}, we see that $N\ci{U}(\xi) = N\ci{V}(\xi)+ \rk \Delta(\xi)=0$ a.e.~$\xi\in\T$.
\end{proof}


\section{Appendix: Derivation of the characteristic function}\label{s:APP}
Mainly for the sake of self-containment, we include a proof of the formulas for the characteristic function in \eqref{e: char funct F1 01} and \eqref{e: char funct F1 02} following that of \cite[Theorem 4.2]{Liaw-Treil_APDE_2019} where the formula was proved in the matrix case. We also prove \eqref{e:F1} at the end of this section.

Recall that for a contraction $T$ its defect operators $D\ci T$ and $D\ci{T^*}$ are given by 
\begin{align*}
D\ci T = (\bI - T^*T)^{1/2}, \qquad D\ci{T^*} = (\bI - TT^*)^{1/2}, 
\end{align*}
and the defect spaces are defined as 
\begin{align*}
\fD\ci{T} = \clos\ran D\ci{T}, \qquad \fD\ci{T^*} = \clos\ran D\ci{T^*}.
\end{align*} 
Recall that according to \cite[Chapter VI]{SzNF2010} the abstract characteristic function $\wt \theta = \wt \theta\ci T$ of the operator $T$ is an analytic function in the unit disc $\D$ whose values are strict contractions $\wt \theta(z) : \fD\ci T\to \fD\ci{T^*}$ which is  given by the formula
\begin{align*}
\wt\theta\ci T (z) = 
(-T + z D\ci{T^*}(\bI\ci\cH-zT^*)^{-1}D\ci{T})\Bigm|_{\fD_T}.
\end{align*}
Usually in the literature the characteristic function is treated as the equivalence class of all functions obtained from $\wt\theta$ by right left  multiplication by constant unitary operators. But sometimes it is more convenient, as we will do now, to pick a concrete representation in this equivalence class. Namely, if $\fD$ and $\fD_*$ are abstract Hilbert spaces of appropriate dimensions, and 
\begin{align}
\label{e: coord operators}
V: \fD\ci{T} \to \fD, \qquad V_*: \fD\ci{T^*} \to \fD_*
\end{align}
are unitary operators (the so-called coordinate operators), then, according to e.g.~\cite[Theorem 1.2.8]{Nik-book-v2} or \cite[Theorem 1.11]{Nik-Vas_model_MSRI_1998}, the representation of the characteristic function corresponding to the identification \eqref{e: coord operators} is given by 
\begin{align}\label{e:theta}
\theta(z) &= V_* \wt\theta (z) V^* = V_* (-T + z D\ci{T^*}(\bI\ci\cH-zT^*)^{-1}D\ci{T})V^*
.
\end{align}

Consider a contraction $T=T\ci\Gamma = U + \bB (\Gamma-\bI)\bB^*U$ from \eqref{e: T_Gamma} where $\bB$ is an isometry acting from  $\fD$ to $\cH$. In this case 
\begin{align}\label{e:defects}
D\ci{T} = U^* \bB D\ci\Gamma \bB^*U, \qquad 
D\ci{T^*} = \bB D\ci{\Gamma^*} \bB^*.
\end{align}
If $\Gamma$ (and therefore $\Gamma^*$) is a strict contraction, the defect spaces are
\begin{align*}
\fD\ci T = \ran (U^*\bB) = U^*\ran\bB, \qquad \fD\ci{T^*} = \ran \bB, 
\end{align*}
so 
\begin{align}
\label{e: coord oper}
V= \bB^*U, \qquad V_* = \bB^*
\end{align}
is a natural choice for the coordinate operators (that is exactly the choice that was made in \cite{Liaw-Treil_APDE_2019}). Note, that in this case $\fD_*=\fD$.

By the definition of $T$ we get using \eqref{e: coord oper} that
\begin{align*}
V_* T V^*  = \bB^*T  U^*\bB\Bigm|_\fD = \bB^*\bB \Gamma\bB^*U  U^*\bB\Bigm|_\fD= \Gamma   ,  
\end{align*}
and therefore \eqref{e:theta} can be rewritten as 
\begin{align}\label{e:char}
\theta(z)
=
-\Gamma+\bB^* z D\ci{T^*}(\bI\ci\cH-zT^*)^{-1}D\ci{T}U^*\bB.
\end{align}
For much of the remainder of this section we include the space on which identity operators are acting on for clarification.

We continue to express the inverse  for $z\in \D$:
\begin{align*}
(\bI\ci\cH-zT^*)^{-1}
&=
\left((\bI\ci\cH-zU^*)[\bI\ci\cH - z (\bI\ci\cH-zU^*)^{-1} U^*\bB (\Gamma^* - \bI\ci\fD) \bB^*]\right)^{-1}
=
    X(z)^{-1}(\bI\ci\cH-zU^*)^{-1}\\
    \intertext{where}
    X(z)&:= \bI\ci\cH - z (\bI\ci\cH-zU^*)^{-1} U^*\bB (\Gamma^* - \bI\ci\fD) \bB^*.
\end{align*}
Here, we note that both $\bI\ci\cH-zT^*$ and $\bI\ci\cH-zU^*$ are invertible for $z\in \D$ (because  $\|zT^*\|, \|zU^*\|\le|z|<1$) so $X(z)$ is invertible as well.
To obtain the expression for $X(z)^{-1}$, we apply Lemma \ref{l:Woodbury} below with $P,Q:\fD\to\cH$ given by
$P=z (\bI\ci\cH-zU^*)^{-1} U^*\bB$ and $Q^*=(\Gamma^* - \bI\ci\fD) \bB^*$  
to get 
\begin{align*}
X(z)^{-1}  &=
    \bI\ci\cH + z(\bI\ci\cH-zU^*)^{-1}U^*\bB\Bigl[\bI\ci\fD - z(\Gamma^* - \bI\ci\fD)\bB^*(\bI\ci\cH-zU^*)^{-1}U^*\bB\Bigr]^{-1}(\Gamma^* - \bI\ci\fD)\bB^* ;
\end{align*}
note that Lemma \ref{l:Woodbury} also implies that the expression in brackets is invertible for $z\in\D$. 

Recalling that $F_1(z) = z\bB^*(\bI\ci\cH - zU^*)^{-1}U^*\bB$ by \eqref{e:F1} we obtain
\begin{align}\label{e:IMinuszTinverse}
(\bI\ci\cH-zT^*)^{-1}
&=
(\bI\ci\cH-zU^*)^{-1}
\notag\\ 
&+
z(\bI\ci\cH-zU^*)^{-1}U^*\bB \Bigl[\bI\ci\fD - (\Gamma^* - \bI\ci\fD)F_1(z) \Bigr]^{-1}(\Gamma^* - \bI\ci\fD)\bB^*(\bI\ci\cH-zU^*)^{-1};
\end{align}
the expression in brackets is invertible for $z\in\D$, because it is just 
the expression in brackets is the above formula for $X(z)^{-1}$. 

Now we substitute \eqref{e:IMinuszTinverse} and \eqref{e:defects} into \eqref{e:char}, and again use that $F_1(z) = z\bB^*(\bI\ci\cH - zU^*)^{-1}U^*\bB$. After straightforward but somewhat tedious calculations we arrive at
\begin{align*}
\theta(z)
&=
-\Gamma+ D\ci{\Gamma^*}\Bigl(F_1(z)+F_1(z)\Bigl[\bI\ci\fD - (\Gamma^* - \bI\ci\fD)F_1(z)\Bigr]^{-1}(\Gamma^* - \bI\ci\fD)F_1(z)\Bigr)D\ci{\Gamma}\\
&=-\Gamma+ D\ci{\Gamma^*}F_1(z)\Bigl[\bI\ci\fD - (\Gamma^* - \bI\ci\fD)F_1(z)\Bigr]^{-1}\Bigl(\bI\ci\fD - (\Gamma^* - \bI\ci\fD)F_1(z)+(\Gamma^* - \bI\ci\fD)F_1(z)\Bigr)D\ci{\Gamma}\\
&= -\Gamma  +   D\ci{\Gamma^*} F_1(z) \Bigl[ \bI\ci\fD -(\Gamma^* - \bI\ci\fD)F_1(z)\Bigr]^{-1}  D\ci\Gamma,
\end{align*}
which is exactly \eqref{e: char funct F1 01}.

Equation \eqref{e: char funct F1 02} is an immediate consequence of \eqref{e: char funct F1 01}. Indeed, we clearly have
\begin{align*}
    \Bigl[\bI\ci\fD -F_1(z)(\Gamma^* - \bI\ci\fD)\Bigr]F_1(z)=
    F_1(z)\Bigl[\bI\ci\fD -(\Gamma^* - \bI\ci\fD)F_1(z)\Bigr],
\end{align*}
or, equivalently,
\begin{align*}
    F_1(z)\Bigl[\bI\ci\fD -(\Gamma^* - \bI\ci\fD)F_1(z)\Bigr]^{-1}=
    \Bigl[\bI\ci\fD -F_1(z)(\Gamma^* - \bI\ci\fD)\Bigr]^{-1}F_1(z).
\end{align*}

The following lemma, which we just used above, can be considered as a particular case of the so-called Woodbury inversion formula, \cite{Wood}, although formally in \cite{Wood} only the case of matrices was treated.

\begin{lm}\label{l:Woodbury}
Let $\cK$ be a separable Hilbert space and consider operators $P,Q: \cK \to \cH$. Operators $\bI\ci\cH - PQ^*$ and $\bI\ci\cK - Q^*P$ are simultaneously invertible. In this case, we have the inversion formula
\begin{align*}
(\bI\ci\cH - PQ^*)^{-1}
=
\bI\ci\cH + P (\bI\ci\cK - Q^*P)^{-1} Q^*.
\end{align*}
\end{lm}

\begin{proof}
Assume that $\bI\ci\cK - Q^*P$ is invertible and compute
\begin{align*}
    (\bI\ci\cH - PQ^*)( \bI\ci\cH & + P (\bI\ci\cK - Q^*P)^{-1} Q^*)\\
    &=
    \bI\ci\cH- PQ^*+ P (\bI\ci\cK - Q^*P)^{-1} Q^* - PQ^*P (\bI\ci\cK - Q^*P)^{-1} Q^*\\
    &=
    \bI\ci\cH+ P 
    \left(-\bI\ci\cK +(\bI\ci\cK - Q^*P)(\bI\ci\cK - Q^*P)^{-1}\right)
    Q^*\\
&=\bI\ci\cH.
\end{align*}
So, $\bI\ci\cH + P (\bI\ci\cK - Q^*P)^{-1} Q^*$ is the right inverse of $\bI\ci\cH - PQ^*.$ To show that it is also the left inverse (and therefore the inverse), one can either reduce $(\bI\ci\cH + P (\bI\ci\cK - Q^*P)^{-1} Q^*)(\bI\ci\cH - PQ^*)$ in analogy, or simply take the adjoint of the above computation and then swap the roles of $P$ and $Q$.

Vice versa, to prove the invertibility of $\bI\ci\cK - Q^*P$ from that of $\bI\ci\cH - PQ^*,$ we simply swap the roles of $P$ and $Q^*$ and those of $\cH$ and $\cK$, respectively, and apply the formulas we just proved.
\end{proof}

\begin{proof}[Proof of Equation \eqref{e:F1}]

This identity follows easily from the definition \eqref{e: bmu 03} of the operator-valued measure $\bmu$: Indeed, since
\begin{align*}
zU^* (\bI-zU^*)^{-1} = (\bI - zU^*)^{-1} -\bI, 
\end{align*}
we see that 
\begin{align*}
z\bB^*(\bI- zU^*)^{-1}U^*\bB & = \bB^* (\bI - zU^*)^{-1} \bB - \bB^*\bB \\
 & = \int_\T \frac{\dd\bmu(\xi)}{1-z\overline \xi} - \int_\T \dd\bmu(\xi) = \int_\T \frac{z\overline \xi}{1-z\overline \xi} \dd\bmu(\xi).
\end{align*}
\end{proof}

\begin{rem*}
Equation \eqref{e:F1} also follows via a ``high brow'' approach invoking the functional calculus. Namely,  for a rational function $\varphi$ (with no poles on $\T$), equation \eqref{e: bmu 03} implies $\bB^*\varphi(U)\bB = \int_\T\varphi(\xi) \dd\bmu(\xi)$. Taking  
\[
\varphi(\xi) =\varphi_z(\xi) =  \frac{z\xi^{-1}}{1-z\xi^{-1}}, 
\]
and using the fact that $\xi^{-1}=\bar\xi$ for $\xi\in\T$, we immediately obtain \eqref{e:F1}. 
\end{rem*}


\providecommand{\bysame}{\leavevmode\hbox to3em{\hrulefill}\thinspace}
\providecommand{\MR}{\relax\ifhmode\unskip\space\fi MR }
\providecommand{\MRhref}[2]{%
	\href{http://www.ams.org/mathscinet-getitem?mr=#1}{#2}
}
\providecommand{\href}[2]{#2}

\end{document}